\documentclass{article}[12pt]%
\usepackage{amsmath}
\usepackage{amsfonts}
\usepackage{amssymb}
\usepackage{graphicx}%

\usepackage{xy} \xyoption{all}
\RequirePackage{color}
\usepackage{multicol}

\newtheorem{theorem}{Theorem}[section]

\newtheorem{corollary}[theorem]{Corollary}

\newtheorem{example}[theorem]{Example}

\newtheorem{lemma}[theorem]{Lemma}

\newtheorem{proposition}[theorem]{Proposition}

\newenvironment{proof}[1][Proof]{\noindent\textbf{#1.} }{\ \rule{0.5em}{0.5em}}
\begin{document}

\title{Existence of maximal ideals in Leavitt path algebras\footnotetext{2010 \textit{Mathematics Subject Classification}: 16D25, 16W50;
\textit{Key words and phrases:} Leavitt path algebras, arbitrary graphs, maximal ideals.}}
\author{Song\"{u}l ES\.{I}N$^{(1)}$, M\"{u}ge KANUN\.{I}$^{(2)}$ \\$^{(1)}$ T\"{u}ccarba\c{s}{\i}  sok. Ka\c{s}e apt. No: 10A/25 Erenk\"{o}y \\Kad\i k\"{o}y, \.{I}stanbul, Turkey.\\E-mail: songulesin@gmail.com\\$^{(2)\text{ }}$Department of Mathematics, D\"{u}zce University, \\Konuralp 81620 D\"{u}zce, Turkey.\\E-mail: mugekanuni@duzce.edu.tr\\}
\date{}
\maketitle

\begin{abstract}
Let $E$ be an arbitrary directed graph and let $L$ be the Leavitt path algebra
of the graph $E$ over a field $K$. The necessary and sufficient conditions are given to assure the existence of a maximal ideal in $L$ and also the necessary and sufficient conditions on the graph which assure that every ideal is contained in a maximal ideal is given.
It is shown that if a maximal ideal $M$ of $L$ is non-graded, 
then the largest graded ideal in $M$, namely $gr(M)$, is also maximal among the graded ideals of $L$. 
Moreover, if $L$ has a unique maximal ideal $M$, then $M$ must be a graded ideal.
The necessary and sufficient conditions on the graph for which every maximal ideal is graded, is discussed.
\end{abstract}


\section{Introduction and Preliminaries} 
It is well-known that in a ring with identity, any ideal is contained in a maximal ideal, however for a non-unital ring the existence of a maximal ideal is not always guarantied. Also, any maximal ideal is not necessarily a prime ideal in a non-unital case. In this short note, for the particular example of a Leavitt path algebra (which is non-unital if the number of vertices of the graph on which it is constructed, is infinite) we discuss the existence of maximal ideals and its characterization via the graph properties.   

The outline of the paper is as follows: 
We give preliminary definitions in the introduction and in Section 2, we discuss maximal and prime ideals in a non-unital ring. Although a Leavitt path algebra does not necessarily have a unit, any maximal ideal is always a prime ideal.  It is shown that if a maximal ideal $M$ of $L$ is non-graded, then the largest graded ideal in $M$, namely $gr(M)$, is also maximal among the graded ideals of $L$. 
In Section 3, first we prove that in a Leavitt path algebra, a maximal ideal exists if and only if there is a maximal hereditary and saturated subset of $E^0$.
We also prove 
the necessary and sufficient condition on the graph which assure that every ideal is contained in a maximal ideal. 
Theorem \ref{UniqueMaxisGraded} states that if $L$ has a unique maximal ideal $M$, then $M$ must be a graded ideal.
Finally, we give the necessary and sufficient conditions on the graph for which every maximal ideal is graded. 

Leavitt path algebras are introduced about a decade ago by \cite{AA} and \cite{AMP} as an algebra defined over a directed graph and since then, it has been a fruitful structure of study because of its close connection to the graph C$^{\ast}$-algebras.

For the general notation, terminology and results in
Leavitt path algebras, we refer to the recently published book \cite{AAS}, and 
for basic results in associative rings and modules, we refer to \cite{AF}. 

A directed graph $E=(E^{0},E^{1},r,s)$ consists of $E^{0}$ the set of \textit{vertices} and
$E^{1}$ the set of \textit{edges}, together with the \textit{range} and the \textit{source} maps $r,s:E^{1}\rightarrow E^{0}$. 

A vertex $v$ is called a \textit{sink} if it emits no edges and a vertex $v$
is called a \textit{regular vertex} if it emits a non-empty finite
set of edges. An \textit{infinite emitter} is a vertex which emits infinitely
many edges. A graph $E$ is called \textit{finite}, if both $E^0$ and $E^1$ are finite, and is called
\textit{row-finite} if the number of edges that each vertex emits is finite. 

For each $e\in E^{1}$, we call $e^{\ast}$ a \textit{ghost edge}. We let
$r(e^{\ast})$ denote $s(e)$, and we let $s(e^{\ast})$ denote $r(e)$. A\textit{
path} $\mu$ of length $|\mu|=n>0$ is a finite sequence of edges $\mu
=e_{1}e_{2}\cdot\cdot\cdot e_{n}$ with $r(e_{i})=s(e_{i+1})$ for all
$i=1,\cdot\cdot\cdot,n-1$. In this case $\mu^{\ast}=e_{n}^{\ast}\cdot
\cdot\cdot e_{2}^{\ast}e_{1}^{\ast}$ is the corresponding ghost path. A vertex
is considered a path of length $0$. The set of all vertices on the path $\mu$
is denoted by $\mu^{0}$.

A path $\mu$ $=e_{1}\dots e_{n}$ in $E$ is \textit{closed} if $r(e_{n})=s(e_{1})$, 
in which case $\mu$ is said to be \textit{based at the vertex
}$s(e_{1})$. A closed path $\mu$ is called \textit{simple} provided that
it does not pass through its base more than once, i.e., $s(e_{i})\neq
s(e_{1})$ for all $i=2,...,n$. The closed path $\mu$ is called a
\textit{cycle} if it does not pass through any of its vertices twice, that is,
if $s(e_{i})\neq s(e_{j})$ for every $i\neq j$.

An \textit{exit }for a path $\mu=e_{1}\dots e_{n}$ is an edge $e$ such that
$s(e)=s(e_{i})$ for some $i$ and $e\neq e_{i}$. We say the graph $E$ satisfies
\textit{Condition (L) }if every cycle in $E$ has an exit. The graph $E$ is
said to satisfy \textit{Condition (K)} if every vertex which is the base of a
closed path $c$ is also a base of another closed path $c^{\prime}$ different
from $c$.
A cycle $c$ in a graph $E$ is called a \textit{cycle without} K, if no vertex
on $c$ is the base of another distinct cycle in $E$ (where distinct cycles
possess different sets of edges and different sets of vertices).
\medskip

If there is a path from vertex $u$ to a vertex $v$, we write $u\geq v$. 

A subset $M$ of $E^0$ is said to be a \textit{maximal tail} if it satisfies the following conditions:  \\
(MT-1) If $v \in M$ and $u \in E^0$ with $u\geq v$, then $u \in M$. \\
(MT-2) If $v \in M$ is a regular vertex, then there exists an edge $e$, with  $s(e)=v$ such that $r(e) \in M$.\\
(MT-3) for any two $u,v\in M$ there exists $w\in M$ such that $u\geq w$ and $v\geq w$. \\
A subset $D$ of $E^0$ 
is called \textit{downward directed} if it satisfies (MT-3) property.    

For any vertex $v$, define $M(v)=\{w\in E^{0}: w\geq v\}.$
\medskip

A subset $H$ of $E^{0}$ is called \textit{hereditary} if, whenever $v\in H$ and
$w\in E^{0}$ satisfy $v\geq w$, then $w\in H$. A hereditary set is
\textit{saturated} if, for any regular vertex $v$, $r(s^{-1}(v))\subseteq H$
implies $v\in H$. The set of all hereditary saturated subsets of $E^{0}$ is denoted by
$\mathcal{H}_E$, which is also a partially ordered set by set inclusion.

Given an arbitrary graph $E$ and a field $K$, the \textit{Leavitt path algebra
}$L_{K}(E)$ is defined to be the $K$-algebra generated by a set $\{v:v\in
E^{0}\}$ of pair-wise orthogonal idempotents together with a set of variables
$\{e,e^{\ast}:e\in E^{1}\}$ which satisfy the following conditions:

\begin{enumerate}
\item[(1)] $s(e)e=e=er(e)$ for all $e\in E^{1}$.

\item[(2)] $r(e)e^{\ast}=e^{\ast}=e^{\ast}s(e)$\ for all $e\in E^{1}$.

\item[(3)] (CK-1 relations) For all $e,f\in E^{1}$, $e^{\ast}e=r(e)$ and
$e^{\ast}f=0$ if $e\neq f$.

\item[(4)] (CK-2 relations) For every regular vertex $v\in E^{0}$,
\[
v=\sum_{\{e\in E^{1},\ s(e)=v\} }ee^{\ast}.
\]

\end{enumerate}

Recall that a ring $R$ is said to have \textit{a set of local units} $F$, where $F$ is a set of idempotents
in $R$ having the property that, for each finite subset $r_1 ,\ldots,r_n$ of $R$, there exists $f \in F$ with $fr_if=r_i$ for
all $1 \leq i \leq n$. 

In the case of Leavitt path algebras, 
for each $x \in L_K(E)$ there exists a finite set of distinct vertices $V(x)$ for which $x = f x f$ , where
$f = \sum_{v \in V(x)} v$. When  $E^0$ is finite, $L_{K}(E)$ is a ring with unit element 
$\displaystyle 1= \sum_{v\in E^{0}}v$. Otherwise, $L_{K}(E)$ is not a unital ring, but is a ring with local units 
consisting of sums of distinct elements of $E^0$.

Every Leavitt path algebra $L_{K}(E)$ is a $\mathbb{Z}$-graded algebra 
$L_{K}(E)={\displaystyle\bigoplus\limits_{n\in\mathbb{Z}}}L_{n}$ 
induced by defining, for all $v\in E^{0}$ and $e\in E^{1}$, $\deg
(v)=0$, $\deg(e)=1$, $\deg(e^{\ast})=-1$. Further, for each 
$n\in \mathbb{Z}$, the homogeneous component $L_{n}$ is given by
\[L_{n}=\left\{{\textstyle\sum}
k_{i}\alpha_{i}\beta_{i}^{\ast}\in L:\text{ }|\alpha_{i}|-|\beta
_{i}|=n\right\}  .
\]
An ideal $I$ of $L_{K}(E)$ is said to be a \textit{graded ideal} if 
$I={\displaystyle\bigoplus\limits_{n\in\mathbb{Z}}}(I\cap L_{n})$.

\medskip

Throughout this paper, $E$ will denote an arbitrary directed graph with no restriction
on the number of vertices and the number of edges emitted by each vertex and
$K$ will denote an arbitrary field. For convenience in notation, we will
denote, most of the times, the Leavitt path algebra $L_{K}(E)$ by $L$. 

All ideals of concern in the sequel will be two-sided ideals. 
In this note, we use both the lattice of all proper ideals, denoted by $\mathcal{L}$, and 
the lattice of all graded proper ideals, denoted by $\mathcal{L}_{gr}$. 
Hence, maximal ideal means a maximal element in $\mathcal{L}$, which can be either a graded or a non-graded ideal; whereas 
a graded maximal ideal is a maximal element in $\mathcal{L}_{gr}$.

\medskip

We shall be using the following concepts and results from \cite{T}. A
\textit{breaking vertex }of a hereditary saturated subset $H$ is an infinite
emitter $w\in E^{0}\backslash H$ with the property that $0<|s^{-1}(w)\cap
r^{-1}(E^{0}\backslash H)|<\infty$. The set of all breaking vertices of $H$ is
denoted by $B_{H}$. For any $v\in B_{H}$, $v^{H}$ denotes the element
$v-\sum_{\{s(e)=v,r(e)\notin H \} }ee^{\ast}$. Given a hereditary saturated subset
$H$ and a subset $S\subseteq B_{H}$, the pair $(H,S)$ is called an \textit{admissible
pair.} The set $\mathbf{H}$ of all admissible pairs becomes a lattice under a
partial order $\leq^{\prime}$ under which $(H_{1},S_{1})\leq^{\prime}
(H_{2},S_{2})$ if $H_{1}\subseteq H_{2}$ and $S_{1}\subseteq H_{2}\cup S_{2}$, (also \cite[Proposition 2.5.6]{AAS}).
Given an admissible pair $(H,S)$, the ideal generated by $H\cup\{v^{H}:v\in
S\}$ is denoted by $I(H,S)$. 

It was shown in \cite{T} that the graded ideals
of $L_{K}(E)$ (i.e. elements of $\mathcal{L}_{gr}$) are precisely the ideals of the form $I(H,S)$ for some
admissible pair $(H,S)$. 
By \cite[Proposition 2.5.4]{AAS},  
for $(H_1, S_1), (H_2, S_2) \in \mathbf{H}$, we have
\[ (H_1, S_1) \leq^{\prime} (H_2, S_2) \iff I(H_1, S_1) \subseteq I(H_2, S_2). \]
In particular, $\subseteq$ is a partial order on  $\mathcal{L}_{gr}$.

\bigskip

Moreover, $L_{K}(E)/I(H,S)\cong L_{K}(E\backslash
(H,S))$. Here $E\backslash(H,S)$ is the \textit{quotient graph of }$E$ in
which 
$(E\backslash(H,S))^{0}=(E^{0}\backslash H)\cup\{v^{\prime}:v\in B_{H}\backslash S\}$ and 
$(E\backslash(H,S))^{1}=\{e\in E^{1}:r(e)\notin H\}\cup\{e^{\prime}:e\in E^{1},r(e)\in B_{H}\backslash S\}$ and
$r,s$ are extended to $(E\backslash(H,S))^{0}$ by setting $s(e^{\prime})=s(e)$
and $r(e^{\prime})=r(e)^{\prime}$. 
The necessary and sufficient condition for all ideals of $L_{K}(E)$ to be graded is that the graph $E$ satisfies condition (K), \cite[Theorem 6.16]{T}. 
On the other hand, a description of non-graded ideals of $L_{K}(E)$ is given in \cite{R-2}.  

A useful observation is that every element $a$ of $L_{K}(E)$ can be written as
$a={\textstyle\sum\limits_{i=1}^{n}}k_{i}\alpha_{i}\beta_{i}^{\ast}$, 
where $k_{i}\in K$, $\alpha_{i},\beta_{i}$
are paths in $E$ and $n$ is a suitable integer. Moreover, 
$L_{K}(E)={\textstyle\bigoplus\limits_{v\in E^{0}}}
L_{K}(E)v={\textstyle\bigoplus\limits_{v\in E^{0}}}vL_{K}(E).$ 
Another useful fact is that if $p^{\ast}q\neq 0$, where $p,q$ are
paths, then either $p=qr$ or $q=ps$ where $r,s$ are suitable paths in $E$.

\medskip 


\section{Maximal ideals versus prime ideals} 
The ring of our concern in this sequel, namely the Leavitt path algebra,
is unital if $E^0$ is finite, and non-unital otherwise. 

In a unital ring, any maximal ideal is also a prime ideal. This is not necessarily true for a non-unital ring as the following example shows.
For completeness of the argument, we will quote the discussion from \cite[pp.86-87]{Ranga1} 
on the maximal and the prime ideals in non-unital Leavitt path algebras.  
\begin{example}
\rm Take the non-unital ring $2\mathbb{Z}$, and its ideal $4\mathbb{Z}$, clearly $4\mathbb{Z}$ is a maximal ideal, but not a prime ideal in $2\mathbb{Z}$.  
\end{example}

\medskip

However, there are non-unital rings where any maximal ideal is also a prime ideal, and a Leavitt path algebra is an example of such.
\begin{proposition}
If $R$ is a ring satisfying $R^2=R$, then any maximal ideal is a prime ideal in $R$.
\end{proposition}
\begin{proof}
Suppose $R^2=R$, and let $M$ be a maximal ideal of $R$ such that $A \nsubseteq M$ and $B \nsubseteq M$ for some ideals $A,B$ of $R$. 
Then $R=R^2=(M+A)(M+B)= M^2 + AM + MB + AB \subseteq M + AB$. Then $M +AB =R$, and $AB \nsubseteq M$. 
Thus $M$ is a prime ideal. 
\end{proof}

As any Leavitt path algebra $L$ is a ring with local units, $L^2 =L$ is satisfied, and the immediate corollary follows. 
\begin{corollary}\label{LPAmaxprime}
In a Leavitt path algebra, any maximal ideal is a prime ideal. 
\end{corollary}
 
Hence, in order to study the maximal ideals of a Leavitt path algebra $L$, we will turn our attention to the prime spectrum of $L$. Prime ideals of Leavitt path algebras has been studied in \cite{Ranga1} and the characterization of both graded and non-graded prime ideals is given. 
We state the result that we will frequently use throughout this note.  

\begin{theorem}(\cite[Theorem 3.12]{Ranga1})\label{RangaPrime} Let $E$ be a graph and $P$ be an ideal of $L$ with $P \cap E^0 =H$. 
Then $P$ is a prime ideal if and only if $P$ satisfies one of the following conditions: \\
(i) $P= I(H, B_H)$ where $E^0 \backslash H$ satisfies MT-3 condition; \\ 
(ii) $P=I(H, S)$ where $S=B_H - \{ u\}$ for some $u \in B_H$ and $E^0 \backslash H=M(u)$; \\
(iii) $P= I(H,B_H) +\langle f(c) \rangle $ where $c$ is a cycle without K in $E$ based at a vertex $u$, 
$E^0 \backslash H=M(u)$ and $f(x)$ is an irreducible polynomial in $K[x,x^{-1}]$. 
\end{theorem}

Recall that for any ideal $N$ of $L$, $gr(N)$ denotes the largest graded ideal contained in $N$. It is proved in \cite[Lemma 3.6]{Ranga1} that 
$gr(N)$ is the ideal generated by the admissible pair $(H,S)$ where $H = N \cap E^0$, and $S = \{ v \in B_H \mid v^H \in N \}$.
Note that if a maximal element $N$ in $\mathcal{L}$ is a graded ideal, then 
clearly $N=gr(N)$ is also a maximal element in $\mathcal{L}_{gr}$. 
We start our discussion by showing that if a maximal element $N$ in $\mathcal{L}$ is non-graded, 
then $gr(N)$ is a maximal element in $\mathcal{L}_{gr}$.

\begin{theorem}\label{gr(M) maximal}If $M$ is a maximal ideal of $L$ which is non-graded,
then $gr(M)$ is maximal among the graded ideals of $L$.   
\end{theorem}

\begin{proof}
Now, $M$ is a prime ideal and so we can write $M=I(H,B_{H})+\langle p(c)\rangle$ where
$H=M\cap E^{0},$ $c$ is a cycle without \textbf{K} in $E$ bases at a vertex
$u$, $p(x)\in K[x,x^{-1}]$ is an irreducible polynomial and $E^{0}\setminus
H=M(u)$ is therefore downward directed. If $N=I(H^{\prime},S^{\prime})$ is a
graded ideal such that $N\supsetneq I(H,B_{H})=gr(M),$ then necessarily
$H^{\prime}\supsetneq H$ and so there is a vertex $v\in H^{\prime}\setminus
H.$ Since $v\geq u,$ $c$ and hence $p(c)\in N.$ This means $M\subsetneqq N$
(note $M\neq N$ as $M$ is non-graded) and, by maximality of $M,N=L.$ Thus
$gr(M)$ is maximal among the graded ideals of $L$.
\end{proof}


\section{Existence of Maximal ideals}

Recall that in a unital ring, it is always true that any ideal is contained in a maximal ideal. Hence, maximal ideals always exist in unital rings. However, this is not necessarily true in a non-unital ring. 
In this section we are going to study when maximal ideals exist and also the conditions on the graph $E$ for which every ideal of $L$ is contained in a maximal ideal. 
We first state an example of a Leavitt path algebra having no maximal ideals. This example is given in \cite[Example p.87]{Ranga1}, \cite[Example 2.8]{EKR}.

\begin{example} \label{ikikulak}
\rm
Let $E$ be a graph with $E^{0}=\{v_{i} :i=1,2,\ldots \}$. For each $i$, there is an edge $e_{i}$ with
$r(e_{i})=v_{i}$ and $s(e_{i})=v_{i+1}$ and at each $v_{i}$ there are two
loops $f_{i},g_{i}$ so that $v_{i}=s(f_{i})=r(f_{i})=s(g_{i})=r(g_{i})$. Thus
$E$ is the graph
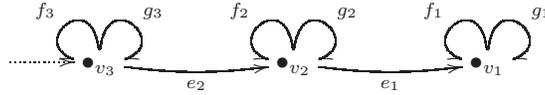
\[ 
\xymatrix{ \ar@{.>}[r] &
\bullet_{v_3}\ar@(u,l)_{f_3} \ar@(u,r)^{g_3} \ar@/_.3pc/[rr]_{e_2} & &  
\bullet_{v_2}\ar@(u,l)_{f_2} \ar@(u,r)^{g_2} \ar@/_.3pc/[rr]_{e_1}  && \bullet_{v_1}\ar@(u,l)_{f_1} \ar@(u,r)^{g_1} }
\]
Clearly $E$ is a row-finite graph and the non-empty proper hereditary
saturated subsets of vertices in $E$ are the sets $H_{n}=\{v_{1},\ldots,v_{n}\}$ 
for some $n\geq1$ and form an infinite chain under set inclusion. 
Graph $E$ satisfies condition (K), so all ideals are graded, generated by $H_n$ for some $n$ and 
they form a chain under set inclusion. 
As the chain of ideals does not terminate, $L_{K}(E)$ does not contain any maximal ideals.
Note also that, $E^{0}\backslash(H_{n},\emptyset)$ is downward directed for each $n$, 
thus all ideals of $L$ are prime ideals.
\end{example}

The necessary and sufficient conditions for the existence of a maximal ideal in a Leavitt path algebra is given in the following lemmata. 
Moreover, the description of the maximal ideals in both graded and non-graded cases are stated.

\begin{lemma}\label{I(H,B) maximal} Let $E$ be any graph and $H \in \mathcal{H}_{E}$. Then $I(H,B_{H})$
is a maximal ideal in $L_{K}(E)$ if and only if $H$ is a maximal element in $\mathcal{H}_{E}$ and the
quotient graph $E\backslash (H,B_{H})$ has Condition $(L)$. 
\end{lemma}

\begin{proof}
Assume $I(H,B_{H})$ is a maximal ideal in $L_{K}(E)$. Notice $L/I(H,B_{H}) \cong L(F)$ is simple where $F=E\backslash (H,B_{H})$. 
By \cite[Theorem 3.1]{AA2}, \cite[Theorem 6.18]{T} the
graph $F$ has Condition (L) and the only hereditary saturated subsets of $F$
are the trivial ones. Suppose $H$ is not a maximal element
in $\mathcal{H}_{E}.$ So there exists $H^{\prime }$ in $\mathcal{H}_{E}$ 
containing $H$. Then $I(H^{\prime },B_{H^{\prime }})/I(H,B_{H})$ is
a proper ideal of $L/I(H,B_{H})$ which contradicts the simplicity.

Conversely, take a maximal element $H\in \mathcal{H}_{E}$ such that  $E\backslash
(H,B_{H})$ has Condition (L). Suppose that there exists a proper ideal $N$
containing $I(H,B_{H})$. Consider the graded part of $N$, $gr(N)=I(H_{1},S_{1})$ 
with $H_{1}\in \mathcal{H}_{E}$ and $S_{1}\subseteq B_{H_{1}}$. Hence
\[ I(H,B_{H})\subseteq gr(N)=I(H_{1},S_{1})\subseteq N.\]
By \cite[Proposition 2.5.4]{AAS}, $(H,B_{H})\leq^{\prime} (H_{1},S_{1})$. Therefore,  
$H\subset H_{1}$ and $B_{H}\subseteq S_{1}\cup H_{1}\subseteq B_{H_{1}}\cup H_{1}.
$ Since $H$ is maximal in $\mathcal{H}_{E},$ it follows that $H=H_{1}$ and
so $B_{H}\subseteq S_{1}\cup H\subseteq B_{H}\cup H$ implies $S_{1}=B_{H}.$
Hence $I(H,B_{H})=gr(N).$ On the other hand by \cite[Theorem 4]{R-2}, $N$ is
of the form 
\[
H\cup \{v^{H}~|~v\in S\subseteq B_{H}\}\cup Y
\]
where $Y$ is a set of mutually orthogonal elements of the form $%
(u+\sum_{i=1}^{n}k_{i}g^{r_{i}})$ in which $g$ is a (unique) cycle with no
exits in $E^{0}\backslash H$ based at a vertex $u$ in $E^{0}\backslash H$
and $k_{i}\in K$ with at least one $k_{i}\neq 0.$ As $E\backslash (H,B_{H})$
has Condition (L), $Y=\emptyset $ and $I(H,B_{H})=N.$ Hence $I(H,B_{H})$ is
a maximal ideal in $L_{K}(E).$
\end{proof}
\medskip

For the existence of a maximal non-graded ideal, $\mathcal{H}_{E}$ needs to have a maximal element whose quotient graph 
does not satify condition (L).  

\begin{lemma}  \label{existNon-graded} For any graph $E$, $H$ is a maximal element in $\mathcal{H}_{E}$ with $E\backslash
(H,B_{H})$ not satisfying Condition (L), if and only if there is a maximal non-graded ideal $M$ containing $I(H,B_H)$ with $H=M \cap E^0$.
\end{lemma}

\begin{proof}
Take a maximal element $H$ in $\mathcal{H}_{E}$ with $E\backslash
(H,B_{H})$ containing a cycle $g$ with no exit and based at a vertex $u$ in $E^{0}\backslash H$. 
Then construct the ideal $N$ generated by 
\[
H\cup \{v^{H}~|~v\in B_{H}\}\cup Y
\]
where $Y$ is a set of mutually orthogonal elements of the form 
$(u+\sum_{i=1}^{n}k_{i}g^{r_{i}})$ in which $k_{i}\in K$ with at least one $k_{i}\neq 0.$ 

Now we want to show that $N$ is not contained in a maximal graded ideal $M$. Since otherwise, 
$M$ is of the form $(H',S')$ where $H'$ is
a hereditary saturated subset of $E^{0}$ and $S'$ is a subset of $B_{H}$. Since every maximal ideal in a Leavitt path
algebra is prime, $M$ is a graded prime ideal. By \cite[Theorem 3.12]{Ranga1}, 
$M=I(H',S')$ with either $S'=B_{H'}$ and $E^{0}\backslash H'$ is MT-3 or $S'=B_{H'}-\{u\}
$ with $E^{0}\backslash H'=M(u)$ for some $u\in B_{H'}$. However, notice that
if $S'=B_{H'}-\{u\}$ for some $u\in B_{H'}$ then 
$M$ cannot be a maximal ideal as it is contained in the proper ideal
generated by $(H',B_{H'})$. So $M=I(H',B_{H'})$. 
Now, by Lemma \ref{I(H,B) maximal}, $H'$ is a maximal element in $\mathcal{H}_{E}$ and 
$E\backslash (H',B_{H'})$ has Condition $(L)$. This contradicts the fact that $H$ is a maximal element and
 $I(H,B_H)=gr(N) \subset N \subsetneq I(H',B_{H'})=M$.

If $N$ is maximal, then we are done. 
If not, then there exists a proper ideal $N_1$ containing $N_0:=N$ properly. 
If $N_1$ is a maximal ideal, then it is a non-graded ideal and we are done. Otherwise, 
continuing in this manner, either there is a maximal non-graded ideal $N_i$ containing $N_{i-1}$ or else there is an infinite chain of proper ideals 
$\{N_i\}_{i \in \mathbb{N} }$, with $N_0=N$, $N_i \subset N_{i+1}$. The graded part of $N_i$, satisfies 
$gr(N_i) \subseteq N_i \subsetneq N_{i+1}$, hence $gr(N_i) \subseteq gr(N_{i+1})$ for all $i$. 
By the maximality of $H$, we conclude that $I(H,B_H)= gr(N)= gr(N_i)$ for all $i$, and $N_i$ is generated by 
$I(H,B_H) \cup \langle f_i(g) \rangle $ for some polynomial $f_i(x) \in K[x]$. 
This yields to a sequence of polynomials $f_i(x)$, with $f_0=f$ and $f_{i+1} | f_i$. 
As there are only finitely many factors of $(1+\sum_{i=1}^{n}k_{i}x^{r_{i}})$, the sequence stabilizes at an 
irreducible polynomial $f$ that divides $(1+\sum_{i=1}^{n}k_{i}x^{r_{i}})$. Hence, $I(H,B_H) \cup \langle f(g) \rangle $ is 
a maximal non-graded ideal. 

Conversely, assume $M$ is a maximal non-graded ideal. By \cite[Theorem 3.12]{Ranga1}, 
$M=I(H,B_{H})+\langle p(c)\rangle$ where $c$ is a cycle without \textbf{K} in $E$ based at
a vertex $u$, $E^{0}\backslash H=M(u)$ and $p(x)$ is an irreducible polynomial in $K[x,x^{-1}]$.
By Theorem \ref{gr(M) maximal}, $I(H,B_{H})$ is a maximal element in $\mathcal{L}_{gr}$. If $H \subseteq H'$, then $I(H',B_{H'})$ contains $I(H,B_H)$, by maximality in $\mathcal{L}_{gr}$, $I(H',B_{H'}) = I(H,B_H)$. Hence, 
$H$ is a maximal element in $\mathcal{H}_{E}$. Now, clearly $E\backslash(H,B_{H})$ cannot satisfy Condition (L), otherwise by Lemma 
\ref{I(H,B) maximal}, $I(H,B_H)$ will be a maximal ideal of $L_K(E)$.
\end{proof}

\bigskip

From Lemma \ref{existNon-graded} and Lemma \ref{I(H,B) maximal}, we deduce that there is a maximal element in $\mathcal{H}_{E}$ if and only if there exists a maximal ideal in $L$. 
We record this result as a theorem below.
Note that the graph in Example \ref{ikikulak} did not have any maximal elements in $\mathcal{H}_E$, hence $L$ did not have any maximal ideals. 
\begin{theorem}\label{existence}(Existence Theorem)
If $L_{K}(E)$ has a maximal ideal  if and only if $\mathcal{H}_E$ has a maximal element.
\end{theorem}
\begin{proof}
By Lemma \ref{existNon-graded} and Lemma \ref{I(H,B) maximal}, the result follows.
\end{proof}

\bigskip

Now, we prove the main result that states the condition when every ideal of a Leavitt path algebra is contained in a maximal ideal.

\begin{theorem}\label{IdealinMaxIdeal}
Assume for every element $X\in \mathcal{H}_{E}$, there exists a maximal
element $Z \in \mathcal{H}_{E}$ such that $X\subseteq Z$, if and only if every ideal in $L_K(E)$ is contained in a maximal ideal.
\end{theorem} 
\begin{proof}
Let $N$ be an ideal, then $N$ is generated by $I(H,B_H) \cup Y$ where $Y$ is a set of mutually orthogonal elements of the form 
$(u+\sum_{i=1}^{n}k_{i}g^{r_{i}})$ in which a cycle $g$ with no
exits in $E^{0}\backslash H$ based at a vertex $u$ in $E^{0}\backslash H$
and $k_{i}\in K$ with at least one $k_{i}\neq 0$. By hypothesis, there is a maximal element $H'$ containing $H$. 

If $E\backslash (H',B_{H'})$ satisfies Condition (L), by Lemma \ref{I(H,B) maximal}, $I(H',B_H')$ is a graded maximal ideal containing $N$. 
If $E\backslash (H',B_{H'})$ does not satisfies Condition (L), by Lemma \ref{existNon-graded}, there is a non-graded maximal ideal 
containing $I(H',B_H')$ containing $N$.

Conversely, take any $X\in \mathcal{H}_{E}$, and form the ideal $I(X,B_X)$. By assumption, $I(X,B_X)$ is contained in a maximal ideal $M$. Moreover, $I(X,B_X) \subseteq gr(M)$.
Then by Lemma \ref{I(H,B) maximal} and Lemma \ref{existNon-graded}, $gr(M)=I(Z,B_Z)$ for some maximal element $Z$ of $\mathcal{H}_{E}$. 
Hence, $X \subset Z$ and the desired result is achieved.
\end{proof}

\medskip

Now, we prove that a unique maximal element in $\mathcal{L}$, has to be a graded ideal. Hence, the unique maximal element in
$\mathcal{L}_{gr}$.  

\begin{theorem}\label{UniqueMaxisGraded}
If $L$ has a unique maximal ideal $M$, then $M$ must be a graded ideal.
\end{theorem}

\begin{proof}
Clearly $M$ is a prime ideal. Assume, by the way of contradiction, that $M$ is
a non-graded ideal with $M\cap E^{0}=H$. Then, by \cite[Theorem 3.12]{Ranga1},
$M=I(H,B_{H})+\langle p(c)\rangle$ where $c$ is a cycle without \textbf{K} in $E$ based at
a vertex $u$, $E^{0}\setminus H=M(u)$ and $p(x)$ is an irreducible polynomial
in $K[x,x^{-1}].$ By Theorem \ref{gr(M) maximal}, $I(H,B_{H})$ is a maximal
graded ideal of $L$. So $L/I(H,B_{H})\cong L_{K}(E\setminus(H,B_{H}))$ has no
proper graded maximal ideal of $L$. So $(E\setminus(H,B_{H}))^{0}
=E^{0}\setminus H$ has no proper non-empty hereditary saturated subsets and so
no proper ideal of $L_{K}(E\setminus(H,B_{H}))$ contains any vertices.
Moreover, $c$ is a cycle without exits based at a vertex $u$ in $E^{0}
\setminus H$ and $v\geq u$ for all $v\in E^{0}\setminus H.$ The last condition
($E^{0}\setminus H=M(u)$) implies that $c$ is the only cycle without exits in
$E^{0}\setminus H.$ Using Lemma 3.5 of \cite{Ranga1}, we then conclude that
every proper ideal of $L/I(H,B_{H})$ is an ideal of the form $\langle f(c) \rangle$ where
$f(x)\in K[x].$ So if $q(x)$ is an irreducible polynomial in $K[x]$ different
from $p(x),$ then $\langle q(c)\rangle$ will be a maximal ideal of $L/I(H,B_{H})$ different
from $\langle p(c)\rangle$. Then $Q=I(H,B_{H})+\langle q(c)\rangle$ will be a maximal ideal of $L$ not
equal to $M$, contradicting the uniqueness of $M$. Hence $M$ must be a graded ideal.
\end{proof}

Let's illustrate the above theorem in different graphs. We start with a finite graph whose corresponding Leavitt path algebra has infinitely many ideals, but only a unique maximal ideal. 

\begin{example}\label{MaxgradedNoK} 
\rm
Let $E$ be the graph 
\medskip

$$ 
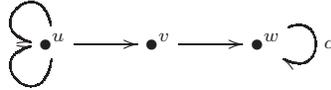
\xymatrix{
{\bullet}^u  \ar@(u,l) \ar@(d,l)   \ar@{->}[r]  & {\bullet}^v \ar@{->}[r]
 & {\bullet}^w \ar@(ur,dr)^{c} }
$$ 
\medskip

Then $E$ does not satisfy Condition (K), so the Leavitt path algebra on $E$ has 
both graded and non-graded ideals. Let $Q$ be the graded ideal generated by the hereditary saturated set $H=\{v,w\}$. $Q$ is a maximal ideal 
as $L/Q$ is isomorphic to the simple Leavitt algebra $L(1,2)$. By using \cite[Theorem 3.12]{Ranga1}, we classify the prime ideals in $L$. There are infinitely many non-graded prime ideals each generated by $f(c)$ where $f(x)$ is an irreducible polynomial in $K[x,x^{-1}]$ which are all contained in $Q$. Also, the trivial ideal $\{ 0\}$ is prime as $E$ satisfies condition MT-3. The prime spectrum of $L$ has a unique maximal element $Q$. 
\end{example}

\medskip

We now give an example of a graph with infinitely many hereditary saturated sets and the corresponding Leavitt path algebra has a unique maximal ideal which is graded.

\begin{example} \label{tersikikulak}
\rm
Let $E$ be a graph with $E^{0}=\{v_{i} :i=1,2,\ldots \}$. For each $i$, there is an edge $e_{i}$ with
$s(e_{i})=v_{i}$ and $r(e_{i})=v_{i+1}$ and at each $v_{i}$ there are two
loops $f_{i},g_{i}$ so that $v_{i}=s(f_{i})=r(f_{i})=s(g_{i})=r(g_{i})$. Thus
$E$ is the graph
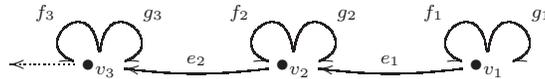
\[ 
\xymatrix{  & \ar@{.>}[l]
\bullet_{v_3}\ar@(u,l)_{f_3} \ar@(u,r)^{g_3} 
& & \bullet_{v_2}\ar@(u,l)_{f_2} \ar@(u,r)^{g_2} \ar@/^.3pc/[ll]_{e_2}  && \bullet_{v_1}\ar@(u,l)_{f_1} \ar@(u,r)^{g_1} \ar@/^.3pc/[ll]_{e_1}}
\]
Now $E$ is a row-finite graph and the non-empty proper hereditary
saturated subsets of vertices in $E$ are the sets $H_{n}=\{v_{n},v_{n+1},\ldots\}$ 
for some $n\geq 2$ and $H_{n+1} \subsetneq H_n$ form an infinite chain under set inclusion and  
$H_{2}=\{v_{2},v_{3},\ldots\}$ is the maximal element in $\mathcal{H}_E$. 
The graph $E$ satisfies Condition (K), so all ideals are graded, generated by $H_n$ for some $n$. 
So $L_{K}(E)$ contains a unique maximal ideal $I(H_2)$. 
Note also that, $E^{0}\backslash H_{n}$ is downward directed for each $n$, 
thus all ideals of $L$ are prime ideals.
\end{example}

\medskip

We pause here to note an obvious corollary which is easily deducible.  

\begin{corollary}
If $I(H,B_{H})$ is a maximal ideal in $L_{K}(E),$ then $E^{0}\backslash H$
satisfies MT-3 condition.
\end{corollary}
\begin{proof}
By Corollary \ref{LPAmaxprime} and Theorem \ref{RangaPrime}, the result follows.
\end{proof}
\medskip

The converse of this corollary is not true. Consider the Example \ref{tersikikulak}, 
and $H_n= \{v_n, v_{n+1}, ... \}$ for $n \geq 3$, $E^0 \backslash H_n = \{v_1, v_2, ..., v_{n-1} \}$ satisfies MT-3 condition.  
(in fact, $E^0 \backslash H_n =M(v_{n-1})$. ) However, $I(H_n)$ for $n \geq 3$ is not a maximal ideal. 


\medskip
We finish this note with the discussion on the necessary and sufficient conditions for all maximal ideals to be graded. In \cite[Corollary 3.13]{Ranga1} it is proved that every prime ideal is graded if and only if the graph $E$ satisfies Condition (K). A natural follow-up question is to determine the conditions on the graph so that every maximal ideal is graded. Example \ref{ikikulak} shows that Condition (K) is not sufficient to assure the existence of any maximal ideals. Moreover, Condition (K) on the graph is not even necessary as Example \ref{MaxgradedNoK} shows. We are ready to state the necessary and sufficient conditions on the graph to assure that all maximal ideals are graded.

\begin{theorem}\label{EveryMaxGraded} Let $E$ be any graph. 
Then every maximal ideal is graded in $L_{K}(E)$ if and only if 
for every maximal element $H$ in $\mathcal{H}_{E}$, $E\backslash
(H,B_{H})$ satisfies Condition (L).
\end{theorem}

\begin{proof} 
$(\Leftarrow )$ Suppose for every maximal element $H$ in $\mathcal{H}_{E}$, $E\backslash
(H,B_{H})$ satisfies Condition (L). Then by Lemma \ref{I(H,B) maximal}, 
$I(H,B_{H})$ is a maximal ideal. However, if there exists a maximal
non-graded ideal $N$, then by Lemma \ref{existNon-graded}, $H:=M \cap E^0$ is a maximal element in 
$\mathcal{H}_{E}$ with $E\backslash (H,B_{H})$ does not satisfies Condition (L) which contradicts the hypothesis.  

$(\Rightarrow )$  
Take a maximal element $H$ in $\mathcal{H}_{E}$ with $E\backslash
(H,B_{H})$ containing a cycle $g$ with no exit. Then 
by Lemma \ref{existNon-graded} there exists a non-graded maximal ideal. 
This gives the required contradiction. 
\end{proof}

\bigskip 

We give an example that has both graded and non-graded maximal ideals to show that converse of 
Theorem \ref{UniqueMaxisGraded} is false.

\begin{example}\label{Maxgraded-nongraded-NoK} 
\rm
Let $E$ be the graph 
\medskip
$$ 
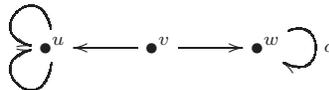
\xymatrix{
{\bullet}^u  \ar@(u,l) \ar@(d,l)    & {\bullet}^v \ar@{->}[r] \ar@{->}[l]
 & {\bullet}^w \ar@(ur,dr)^{c} }
$$ 
\medskip

Then the Leavitt path algebra on $E$ has both graded and non-graded maximal ideals. The set $\mathcal{H}_E$ is finite and hence any ideal is contained in a maximal ideal. The trivial ideal $\{ 0\}$ which is a graded ideal generated by the empty set, is not prime as $E$ does not satisfy condition MT-3. There are infinitely many non-graded prime ideals each generated by $f(c)$ where $f(x)$ is an irreducible polynomial in $K[x,x^{-1}]$ which all contain $\{ 0\}$.
Let $N$ be the graded ideal generated by the hereditary saturated set $H=\{u\}$ and in this case, the quotient graph $E \backslash H$ does not satisfy condition (L). By Lemma \ref{existNon-graded}, there are infinitely many maximal non-graded ideals each generated by $f(c)$ where $f(x)$ is an irreducible polynomial in $K[x,x^{-1}]$ which all contain $N$. 
Also, let $Q$ be the graded ideal generated by the hereditary saturated set $H=\{w\}$. In this case, the quotient graph $E \backslash H$ satisfy condition (L). By Lemma \ref{I(H,B) maximal}, $Q$ is a maximal ideal. 
The prime spectrum of $L$ has a infinitely many maximal ideals, one of them is graded, namely $Q$ and infinitely many are non-graded ideals whose graded part is $N$. 
\end{example}

{\bf Acknowledgement}
The authors are deeply thankful to Professor K.M. Rangaswamy for bringing up this
research question, for his helpful discussion and improvement in the proofs of Theorems \ref{gr(M) maximal}, and 
\ref{UniqueMaxisGraded}. Both authors also thank Duzce University Distance Learning Center (UZEM) for using its facilities while conducting this research.

\end{document}